\documentclass[12pt,a4paper]{article}

\usepackage{amsmath,amsthm,amssymb}
\usepackage[truedimen,margin=25truemm]{geometry}
\usepackage[dvipdfmx]{graphicx}
\usepackage{subfigure}

\theoremstyle{plain}

\newtheorem{definition}{Definition}
\newtheorem{thm}[definition]{Theorem}

\newtheorem{lem}[definition]{Lemma}

\def\G{\Gamma}
\def\g{\gamma}
\def\Pe{L}
\DeclareMathOperator{\Ex}{{E}}
\DeclareMathOperator{\Va}{{V}}

\DeclareMathOperator{\sgn}{{sgn}}

\begin{document}
\title{Inequality for the variance of an asymmetric loss}
\author{Naoya Yamaguchi, Yuka Yamaguchi and Maiya Hori}
\date{\today}

\maketitle

\begin{abstract}
We assume that the forecast error follows a probability distribution which is symmetric and monotonically non-increasing on non-negative real numbers, 
and if there is a mismatch between observed and predicted value, then we suffer a loss. 
Under the assumptions, we solve a minimization problem with an asymmetric loss function. 
In addition, we give an inequality for the variance of the loss. 
\end{abstract}

\section{Introduction}
Let $\hat{y}$ be a predicted value of an observed value $y$. 
In this paper, 
we make the assumptions~(I) and (II): 
\begin{enumerate}
\item[(I)] 
The prediction error $z := \hat{y} - y$ is the realized value of a random variable $Z$, 
whose probability density function $f(z)$ satisfies $f(x) = f(- x)$ for $x \in \mathbb{R}$ and $f(x) \geq f(y)$ for $0 \leq x \leq y$. 
\item[(II)]  Let $k_{1}$, $k_{2} \in \mathbb{R}_{> 0}$. 
If there is a mismatch between $y$ and $\hat{y}$, 
then we suffer a loss 
\begin{align*}
\Pe(z) := 
\begin{cases}
k_{1} z, & z \geq 0, \\ 
- k_{2} z, & z < 0. 
\end{cases}
\end{align*}
\end{enumerate}
Under the assumptions~(I) and (II), we solve the minimization problem for the expected value of $L(Z + c)$: 
$$
C = \arg{\min_{c}}\{ \Ex[L(Z + c)] \}. 
$$
In addition, we give the following theorem. 
\begin{thm}\label{thm:1}
We have 
$$
\Va[L(Z + C)] \leq \Va[L(Z)], 
$$
where equality holds only when $C = 0$; that is, when $k_{1} = k_{2}$. 
\end{thm}

Theorem~$\ref{thm:1}$ is obtained by the following lemma. 

\begin{lem}\label{lem:2}
Suppose that a probability density function $f(t)$ is monotonically non-increasing on $\mathbb{R}_{\geq 0}$ and satisfies $\int_{0}^{\infty} f(t) dt = \frac{1}{2}$. 
Then, 
for any $x \geq 0$, we have 
\begin{align*}
\alpha(x) := 4 \int_{0}^{x} f(t) dt \int_{x}^{\infty} t f(t) dt - \frac{x}{2} + 2 x \left( \int_{0}^{x} f(t) dt \right)^{2} \geq 0. 
\end{align*}
If $f(t)$ is strictly decreasing, then $\alpha(x) > 0$ holds for $x > 0$. 
Also, $\alpha(x) = 0$ holds for $x \geq 0$ if and only if $f(t)$ equals to the probability density function of a continuous uniform distribution on $\mathbb{R}_{\geq 0}$. 
\end{lem}
These results are a generalization of the results of \cite{doi:10.1080/02664763.2020.1761951}. 
The paper~\cite{doi:10.1080/02664763.2020.1761951} made the assumptions~(I') and (II): 
\begin{enumerate}
\item[(I')] 
The prediction error $z := \hat{y} - y$ is the realized value of a random variable $Z$, 
whose probability density function is a generalized Gaussian distribution function (see, e.g., \cite{Dytso2018}, \cite{doi:10.1080/02664760500079464}, and \cite{Sub23}) with mean zero 
\begin{align*}
f(z) := \frac{1}{2 a b \G(a)} \exp{\left( - \left\lvert \frac{z}{b} \right\rvert^{\frac{1}{a}} \right)}, 
\end{align*}
where $\G(a)$ is the gamma function and $a, b > 0$. 
\end{enumerate}
Assumption~(I) is weaker than (I'). 
Thus, we assume a more general situation than in \cite{doi:10.1080/02664763.2020.1761951}. 
In \cite{doi:10.1080/02664763.2020.1761951}, 
under the assumptions~(I') and (II), 
the minimization problem for the expected value of $L(Z + c)$ is solved and the inequality $\Va[L(Z + C)] \leq \Va[L(Z)]$ is obtained. 
This inequality is derived from the following inequality: {\it For $a, x > 0$, we have 
\begin{align}
x^{a} \g(a, x)^{2} - x^{a} \G(a)^{2} + 2 \g(a, x) \G(2a, x) > 0, 
\end{align}
where }
\begin{align*}
\G(a) := \int_{0}^{+\infty} t^{a - 1} e^{- t} dt, \quad \G(a, x) &:= \int_{x}^{+\infty} t^{a - 1} e^{-t} dt, \quad \g(a, x) := \int_{0}^{x} t^{a - 1} e^{-t} dt. 
\end{align*}
Inequality~(1) is the special case of Lemma~$\ref{lem:2}$ that $f(z)$ is a generalized Gaussian distribution function. 

Assumptions~(I) and (II) have a background in the procurement from an electricity market. 
Suppose that we purchase electricity $\hat{y}$ from an market, based on a forecast of the electricity $y$ that will be needed. 
This situation makes the assumption~(I). 
If $\hat{y} - y > 0$, then there is a waste of procurement fee proportional to $\hat{y} - y$. 
If $y - \hat{y} > 0$, then we are charged with a penalty proportional to $y - \hat{y}$. 
This situation makes the assumption~(II). 
For details, see \cite{MR3850086}.

\section{Proof of results}
For $c \in \mathbb{R}$, let $\sgn(c) := 1 \: (c \geq 0)$; $- 1 \: (c < 0)$. 
From $\int_{0}^{\infty} f(z) dz = \frac{1}{2}$, 
the expected value of $L(Z + c)$ and $L(Z + c)^{2}$ are as follows: For any $c \in \mathbb{R}$, 
\begin{align*}
\Ex[ L(Z + c) ] 
&= (k_{1} + k_{2}) \int_{|c|}^{\infty} z f(z) dz + \frac{c (k_{1} - k _{2})}{2} + |c| (k_{1} + k_{2}) \int_{0}^{|c|} f(z) dz, \\ 
\Ex[ L(Z + c)^{2} ] 
&= (k_{1}^{2} + k_{2}^{2}) \int_{0}^{\infty} z^{2} f(z) dz + \sgn(c) (k_{1}^{2}- k_{2}^{2}) \int_{0}^{|c|} z^{2} f(z) dz \\ 
&\qquad + 2 c (k_{1}^{2} - k_{2}^{2}) \int_{|c|}^{\infty} z f(z) dz + \frac{c^{2} (k_{1}^{2} + k_{2}^{2})}{2} + c |c| (k_{1}^{2} - k_{2}^{2}) \int_{0}^{|c|} f(z) dz. \\ 
\end{align*}
Therefore, the expected value and the variance of $L(Z)$ are as follows: 
\begin{align*}
\Ex[ L(Z) ] &= (k_{1} + k_{2}) \int_{0}^{\infty} z f(z) dz, \\ 
\Va[ L(Z) ] &= (k_{1}^{2} + k_{2}^{2}) \int_{0}^{\infty} z^{2} f(z) dz - (k_{1} + k_{2})^{2} \left( \int_{0}^{\infty} z f(z) dz \right)^{2}. 
\end{align*}

We determine the value $c$ that gives the minimum value of $\Ex[L(Z + c)]$. 
From 
\begin{align*}
\frac{d}{dc} \Ex[L(Z + c)] &= \frac{k_{1} - k_{2}}{2} + \sgn(c) (k_{1} + k_{2}) \int_{0}^{|c|} f(z) dz, \\ 
\frac{d^{2}}{dc^{2}} \Ex[L(Z + c)] &= (k_{1} + k_{2}) f(c) \geq 0, 
\end{align*}
we can see that $\Ex[L(Z + c)]$ has the minimum value  at the zero point of $\frac{d}{dc} \Ex[L(Z + c)]$. 
The zero point $C$ satisfies the following equation: 
$$
\frac{k_{1} - k_{2}}{2} + \sgn(C) (k_{1} + k_{2}) \int_{0}^{|C|} f(z) dz = 0. 
$$
From this, 
$C = 0$ if and only if $k_{1} = k_{2}$. 
Also, we have 
\begin{align*}
\Ex[L(Z + C)] &= (k_{1} + k_{2}) \int_{|C|}^{\infty} z f(z) dz, \\ 
\Va[L(Z + C)] 
&= (k_{1}^{2} + k_{2}^{2}) \int_{0}^{\infty} z^{2} f(z) dz - 2 (k_{1} + k_{2})^{2} \int_{0}^{|C|} f(z) dz \int_{0}^{|C|} z^{2} f(z) dz \\ 
&\qquad- 4 |C| (k_{1} + k_{2})^{2} \int_{0}^{|C|} f(z) dz \int_{|C|}^{\infty} z f(z) dz + \frac{C^{2} (k_{1} + k_{2})^{2}}{4} \\ 
&\qquad \qquad - (k_{1} + k_{2})^{2} \left( \int_{|C|}^{\infty} z f(z) dz \right)^{2} - C^{2} (k_{1} + k_{2})^{2} \left( \int_{0}^{|C|} f(z) dz \right)^{2}. 
\end{align*}
Let 
\begin{align*}
\beta(x) 
&:= - \left( \int_{0}^{\infty} z f(z) dz \right)^{2} + 2 \int_{0}^{x} f(z) dz \int_{0}^{x} z^{2} f(z) dz + 4 x \int_{0}^{x} f(z) dz \int_{x}^{\infty} z f(z) dz \\ 
&\qquad - \frac{x^{2}}{4} + \left( \int_{x}^{\infty} z f(z) dz \right)^{2} + x^{2} \left( \int_{0}^{x} f(z) dz \right)^{2}. 
\end{align*}
Then, $\Va[L(Z)] - \Va[L(Z + C)] = (k_{1} + k_{2})^{2} \beta(C)$ holds. 
From $\beta(0) = 0$ and 
\begin{align*}
\frac{d}{dx} \beta(x) &= 4 \int_{0}^{x} f(z) dz \int_{x}^{\infty} z f(z) dz - \frac{x}{2} + 2 x \left( \int_{0}^{x} f(z) dz \right)^{2} \\
&\qquad + 2 f(x) \int_{0}^{x} z^{2} f(z) dz + 2 x f(x) \int_{x}^{\infty} z f(z) dz, 
\end{align*}
if Lemma~$\ref{lem:2}$ is proved, 
then Theorem~$\ref{thm:1}$ is immediately obtained. 
We prove Lemma~$\ref{lem:2}$. 

\begin{proof}[Proof of Lemma~$\ref{lem:2}$]
Take any $x \geq 0$. 
If $f(x) = 0$, then $\alpha(x) = 0 - \frac{x}{2} + 2 x \cdot \frac{1}{4} = 0$. 
Below, we consider the case that $f(x) > 0$. 
Let $\gamma := \int_{0}^{x} f(t) dt$. 
For a function $g = g(t)$ satisfying $f(x) \geq g(t) \geq 0$ for $x \leq t$ and $\gamma + \int_{x}^{\infty} g(t) dt = \frac{1}{2}$, 
we define a functional $S(g)$ by 
\begin{align*}
S(g) := \int_{x}^{\infty} t g(t) dt. 
\end{align*}
Regarding $S(g)$ as a solid with the bottom surface area $\int_{x}^{\infty} g(t) dt = \frac{1}{2} - \gamma$ (constant), 
we find that if we make $g(t)$ as large as possible within the range where $t$ is small, 
then $S(g)$ become smaller. 
Thus, the function $g$ that minimizes $S(g)$ is $g(t) = u(t)$ defined by 
\begin{align*}
u(t) := 
\begin{cases}
f(x), & x \leq t \leq x + \frac{1}{f(x)} \left( \frac{1}{2} - \gamma \right), \\ 
0, & \text{otherwise}. 
\end{cases}
\end{align*}
From 
\begin{align*}
S(u) = \int_{x}^{\infty} t u(t) dt = x \left( \frac{1}{2} - \gamma \right) + \frac{1}{2 f(x)} \left( \gamma^{2} - \gamma + \frac{1}{4} \right)
\end{align*}
and $\gamma \geq x f(x)$, we have 
\begin{align*}
\alpha(x) 
&\geq 4 \gamma S(u) - \frac{x}{2} + 2 x \gamma^{2} \\ 
&= 4 \gamma \left\{ x \left( \frac{1}{2} - \gamma \right) + \frac{1}{2 f(x)} \left( \gamma^{2} - \gamma + \frac{1}{4} \right) \right\} - \frac{x}{2} + 2 x \gamma^{2} \\ 
&\geq 2 x \gamma - 4 x \gamma^{2} + 2 x \left( \gamma^{2} - \gamma + \frac{1}{4} \right) - \frac{x}{2} + 2 x \gamma^{2} \\ 
&= 0. 
\end{align*}
Also, from this, if $f(t)$ is strictly decreasing, then $\alpha(x) > 0$ holds for $x > 0$. 
In addition, $f(t)$ is the function of the form 
\begin{align*}
f(t) = 
\begin{cases}
\frac{1}{2 a}, & 0\leq t \leq a, \\ 
0, & t > a
\end{cases}
\end{align*}
if and only if $\alpha(x) = 0$ holds for $x \geq 0$. 
\end{proof}

\clearpage

\bibliography{reference}
\bibliographystyle{plain}

\medskip
\begin{flushleft}
Faculty of Education, 
University of Miyazaki, 
1-1 Gakuen Kibanadai-nishi, 
Miyazaki 889-2192, 
Japan \\ 
{\it Email address}, Naoya Yamaguchi: n-yamaguchi@cc.miyazaki-u.ac.jp \\ 
{\it Email address}, Yuka Yamaguchi: y-yamaguchi@cc.miyazaki-u.ac.jp \\ 
\end{flushleft}

\begin{flushleft}
General Education Center, Tottori University of Environmental Studies,
1-1-1 Wakabadai-kita, Tottori, 689-1111. Japan \\ 
{\it Email address}, Maiya Hori: m-hori@kankyo-u.ac.jp
\end{flushleft}

\end{document}